\documentclass[11pt]{article}
\usepackage{amssymb,amsmath,amsthm,amsfonts}
\usepackage[dvips]{graphicx}
\usepackage{graphicx}
\usepackage[T1]{fontenc}

\def\disp{\displaystyle}

\def\crr{\cr\noalign{\vskip2mm}}

\def\dref#1{(\ref{#1})}

\theoremstyle{plain}
\newtheorem{theorem}{Theorem}[section]
\newtheorem{lemma}{Lemma}[section]

\numberwithin{equation}{section}

\theoremstyle{definition}

\newtheorem{remark}{Remark}[section]

\begin{document}
\title{{\bf A unique continuation  property for a class of parabolic differential inequalities in a bounded domain\footnote{\small
  This work
was partially supported by  the Natural Science
Foundation of Henan Province (No. 162300410176).
 Corresponding Author: Taige Wang, Email: wang2te@ucmail.uc.edu.  }} }

\author{Guojie Zheng$^{a}$, \;Dihong Xu$^{b}$,   \;Taige Wang$^{c}$
\\
$^a${\it College of Mathematics and Information Science}\\
{\it Henan Normal University, Xinxiang, 453007, P.R. China}\\
$^b${\it College of Engineering}\\
{\it Huazhong Agricultural Univetsity, Wuhan, 430070, P.R. China}\\
$^c${\it Department of Mathematical Sciences}\\
{\it  University of Cincinnati, Cincinnati, OH 45221, USA.}}

\date{}

 \maketitle

\begin{abstract}
This article is concerned with the unique continuation property of a forward in-time differential inequality abstracted from parabolic equations proposed on a convex domain $\Omega$ prescribed with homogeneous Dirichlet boundary conditions. Our result shows that the value of the solutions can be determined uniquely by its value on an arbitrary open subset $\omega$ in $\Omega$ at any given positive time $T$.
 We also derive the quantitative nature of this unique continuation, that is, the estimate of a Sobolev norm of the initial data on $\Omega$, which is majorized by that of solution on the bounded open subset $\omega$ at terminal moment $t=T$. \\

\vspace{0.3cm}

\vspace{0.3cm}


\vspace{0.3cm}

\end{abstract}

\section{Introduction}\label{sec1}
\ \ \ \  Suppose that $\Omega$ is a convex bounded region in $\mathbb{R}^n\,\,(n\ge 1)$ with a smooth boundary $\partial\Omega$.
Let $T$ be a given positive constant. We consider a forward differential inequality in time with 1st order derivative term, which reads:
\begin{equation}\label{1.1}
|\partial_t u(x,t)-\triangle u(x,t)|\leq M(|\nabla u(x,t)|+|u(x,t)|),
\end{equation}
in  $(x,t)\in Q=\Omega\times(0,T]$, where $M$ is a positive number. In this paper, we will discus the unique continuation  property for solution of (\ref{1.1}) under suitable regularity assumption on $u$.\\

 Unique continuation property of solutions is an interesting topic related to inverse problems and observability of control theories of PDEs, and it was first found to hold for elliptic equations as it naturally holds for harmonic functions, then for some classes of parabolic equations. The first result about unique continuation of strong solutions of a parabolic equation with constant coefficients is in \cite{Landis}, where E. Landis and O. Oleinik used reduction, from study of parabolic equations extended to elliptic ones, for the original parabolic equations with time-invariable coefficients. To prove unique continuation, two methodologies, \emph{Carleman inequalities} and \emph{frequency functions}, are involved; the former one is widely used as an effective tool to obtain estimates of frequency functions, when tackling time-variable coefficients; while property of the latter one can be reached via less complicated calculation than Carleman's, only if constant coefficient cases are discussed. The interplay of Carleman estimates and frequency functions are described clearly in L. Escauriaza et al \cite{EscauriazaF1}. Similar results for heat operators and parabolic-type equations via Carleman's can be found in C. Escauriaza, F. Fern\'{a}ndez, C. Kenig, G. Seregin, V. \u{S}ver\'{a}k, D. Tataru, L. Vega's works \cite{C. Escauriaza, Sverak, EscauriazaF2, Fernandez,  C. Kenig,  H. Koch} while frequency function method can be found in works by F. Almgren, N. Garofalo, F. Lin, C. Poon, K. Phung, G. Wang et al (see, e.g., \cite{Almgren, Lin2, F. Lin, Q. Lu, Poon, Phung0, Phung1, Phung2, Phung3}). \\
 
Backward-in-time inequalities derived from abstract parabolic equations were considered in $Q_{R, T} = \mathbb{R}^n \textbackslash  B_R\times [0, T]$ by L. Escauriaza, G. Seregin, and V. \u{S}ver\'{a}k (see, e.g., \cite{Sverak, Seregin, ESS}), which are insightful for regularity results of solutions of incompressible Navier-Stokes equations posed in $\mathbb{R}^3$. Therein, backward operator $\partial_t + \triangle$ is in inequalities and backward uniqueness is considered, which is usually referred to a weaker result than unique continuation, as property on overall domain controlled by that of a real open subset is absent. In \cite{Seregin}, $M=0$, inequality is not null-controllable by any boundary control applied on boundary of ball $\partial B_R$. In \cite{Sverak}, authors proposed a growth condition, $|u(x,t)|\le Me^{M|x|^2}$ for some positive $M$ and any $t>0$, which can't be weakened to obtain the backwardness, then the regularity of solution is obtained. In these works which are carried out by using Carleman inequalities, they also found, null controllability doesn't exist for any bounded controls $b(x,t)$ and $c(x,t)$ in backward equation 

$$\partial_tu + \triangle u + b\cdot\nabla u + cu=0.$$

Moreover, when $n=3$, weak solution $v(x,t)$ of incompressible N-S equation is smooth under assumption $v(x,t)\in L^{\infty}(0,T;L^3(\mathbb{R}^3))$. In \cite{ESS}, unique continuation and backward uniqueness are both pursued on $B_R\times(0,T)$ with $n=3$ and $\mathbb{R}_+^{n}\times(0,T)$ for any dimension $n$, respectively. Note, in \cite{ESS}, same growth condition in \cite{Sverak} is used as well for proving backward uniqueness. \\

Recently, G. Camliyurt and I. Kukavica in \cite{Camliyurt} proved the unique continuation by abtaining finite order of vanishing for forward parabolic PDE with 1st derivative terms, whose coefficients are variable and bounded. Frequency functions and a technique of changing variables are invoked in their work. This situation would be a case of differential inequality discussed in this paper. We refer readers to similar discussion on elliptic equations by H. Donnelly and C. Fefferman \cite{DonnellyF} and I. Kukavica \cite{Kukavica} for motivations in early years.\\ 

Inequalities with less smooth variable coefficients can be considered in plenty of seminal literatures, including that for dispersive PDEs such as linear Schr\"{o}dinger equations. Among them, $|\partial_t u+\triangle u|\le |V(x,t)u|$ in $\Omega\times[0,T]$ is discussed in C. Sogge's \cite{Sogge} and C. Escauriaza and F. Fern\'{a}ndez \cite{EscauriazaF2}. In \cite{Sogge}, unbounded potential $V(x,t)\in L_{loc}^{n+2\over2}(dxdt)$ is assumed, and the author obtained weak unique continuation. Strong uniqueness continuation was obtained by D. Jerison and C. Kenig for Schr\"{o}dinger operators with $V\in L_{loc}^{n\over 3}(dxdt)$ (see, e.g., \cite{Jerison}), which is proved to be sharp. In \cite{EscauriazaF2}, $V$ is bounded hence the inequalities are similar to ours but backward. We would also refer readers to L. Escauriaza and L. Vega's \cite{C. Escauriaza, EscauriazaVega} for results about heat operators with other different conditions restricted on potential $V$ and references therein. \\

Most studies of unique continuation / backward uniqueness of differential inequalities are carried out owing to Carleman inequalities, while this paper will follow the clue provided by frequency function to pursue the strong unique continuation.\\

To facilitate our discussion,  we make the following conditions for throughout the paper: 

\vspace{0.3cm}
\noindent {\bf Assumption 1.} The regularity for solution of (\ref{1.1}) is
 $$u(x,t)\in L^2(0,T; H_0^1(\Omega))\cap H^1(0,T; H^{-1}(\Omega)).$$

\noindent {\bf Assumption 2.} The growth condition for $u$ is
\begin{eqnarray*}
\int_{\Omega }|u(x, T)|^2dx\le e^{CM(T-t)}\int_{\Omega} |u(x,t)|^2dx,
\end{eqnarray*}
for any $t\in[0,T]$.\\

\noindent {\bf Assumption 3.} Suppose that $\partial_t u(x,t)-\triangle u(x,t)\in L^2(\Omega)$, and
  \begin{eqnarray*}
\|\partial_t u(x,t)-\triangle u(x,t)\|_{H^{-1}(\Omega)} \le CM\|u(x,t)\|_{L^2(\Omega)},
\end{eqnarray*}
 for $t\in(0,T]$ a.e..\\
\vspace{0.3cm}

Throughout the rest of the paper, the following notation will be used. We denote $\|\cdot\|_X$ to the norm of a Banach space $X$, and
$\langle\cdot,\cdot\rangle$, to the inner product of $L^2(\Omega)$ respectively.
 Besides, variables $x$ and $t$ for functions of $(x, t)$ and variable
$x$ for functions of $x$ will be omitted, provided that it does cause some confusion.
Let $\omega\subset\Omega$ be an nonempty and open subset of $\Omega$. The unique continuation property obtained are stated as follows:
\begin{theorem}\label{theorem1.1}Suppose that Assumptions 1, 2 hold. Then,
 there are positive numbers: $\gamma= C(\Omega, \omega, T)$ and $C=C(\Omega,\omega)$  such that, any solution $u$ of equation (\ref{1.1}) has the following estimate:
\begin{eqnarray}\label{1.4}
\int_{\Omega}|u(x,T)|^2dx&\leq& C\exp\left(\frac{C}{T}+C(MT+M^2T^2)\right)\crr\disp
&&\times\left(\int_{\Omega}|u_0(x)|^2dx\right)^{1-\gamma}\times\left(\int_{\omega}|u(x,T)|^2dx\right)^\gamma.
\end{eqnarray}
\end{theorem}
\begin{remark}
$(i)$ The constant $C$ in (\ref{1.4}) or (\ref{1.5}) stands for a positive
constant only dependent on domains $\Omega$ and  $\omega$. This constant varies in different contexts.\\

\noindent $(ii)$ This result demonstrates that  solutions of (\ref{1.1}) can be  uniquely determined by its value on an open subset $\omega$ at any given positive time $T$. It also shows that the solutions of (\ref{1.1}) must vanish if it vanishes in an open subset $\omega$ at time $T$.
\end{remark}
\begin{theorem}\label{theorem1.3}
Suppose that Assumptions 1-3 hold.
 If $u(x,0)\not\equiv0$, then, there exists a positive number  $C=C(\Omega,\omega)$ such that solution $u$   of \dref{1.1} has the following estimate:
\begin{eqnarray}\label{1.5}
\int_{\Omega}|u(x,0)|^2dx\!\leq\!C\exp\left(C(\frac{1}{T}+1+MT+M^2T^2)e^{CM^2T}\!{\|u_0\|_{L^2(\Omega)}^2\over \|u_0\|_{H^{-1}(\Omega)}^2}\right)
\!\times\! \int_{\omega}(|u(x,T)|^2)dx.
\end{eqnarray}
\end{theorem}
Uniqueness refers to the fact of that of initial states, given the observations at terminal are same. Under some circumstances, it suffices to derive similar inequality only about $u$ instead of comparing difference as the observation in a subset of $\Omega$ at terminal $t=T$.   Hence, it is meaningful to consider the unique continuation, if norm of the data at the terminal (i.e., solution at $t=T$) is bounded in certain function spaces. \\

We organize the paper as follows. In section 2, some preliminary results are
presented.  Section 3 is devoted to  the unique continuation  property for the solution of  (\ref{1.1}).

\section{Preliminary lemmas}

Given a positive number $\lambda$, we  define
\begin{eqnarray}\label{G}
G_\lambda(x,t)=\frac{1}{(T-t+\lambda)^{n/2}}e^{-\frac{|x-x_0|^2}{4(T-t+\lambda)}},\:\: (x,t)\in \Omega\times(0,T),
\end{eqnarray} where $x_0\in \Omega.$ $G_{\lambda}$ is referred as a \emph{caloric function}(see, e.g., \cite{EscauriazaF1}). \\

Then, for each  $t\in[0,T]$, we define functions of time by solution $u(x, t)$ of first equation in system (\ref{1.1}):
\begin{eqnarray}\label{H}
H_\lambda(t)=\int_\Omega|u(x,t)|^2G_\lambda(x,t)dx,
\end{eqnarray}
\begin{eqnarray}\label{D}
D_\lambda(t)=\int_\Omega|\nabla u(x,t)|^2G_\lambda(x,t)dx,
\end{eqnarray}
and therefore, frequency function is defined as
\begin{eqnarray}\label{N}
N_\lambda(t)=\frac{2D_\lambda(t)}{H_\lambda(t)}.
\end{eqnarray}
$N_\lambda(t)$ was first discussed in \cite{Almgren}, and sequentially in \cite{C. Escauriaza, Lin2, Poon}.
We have $H_\lambda(t)\neq0$ at any moment throughout the paper.\\

Next, we will discuss the properties for the functions $G_\lambda(x,t)$, $H_\lambda(t)$, $D_\lambda(t)$, and $N_\lambda(t)$. Note that $G_\lambda(t)$ behaves like a backward heat kernel. Subset $\omega$ lying compactly in $\Omega$ includes circles, and we pick such an open ball $B_r$ centered at $x_0\in \omega$ with radius $r$, that is, $B_r\subset\omega$. $B_r$ denotes the open ball. Let $m={\disp\textrm{sup}_{x\in\Omega}}|x-x_0|^2$. The following Lemma \ref{lemma4.1} is directly borrowed from \cite{C. Escauriaza, Phung1}.
\begin{lemma}\label{lemma4.1}
For $\lambda>0$, the function $G_\lambda(x,t)$ holds the following four identities over $\mathbb{R}^n\times[0,T]$:
\begin{eqnarray}\label{4.5}
\partial_tG_\lambda(x,t)+\triangle G_\lambda(x,t)=0,
\end{eqnarray}
\begin{eqnarray}\label{4.6}
\nabla G_\lambda(x,t)=\frac{-(x-x_0)}{2(T-t+\lambda)}G_\lambda(x,t),
\end{eqnarray}
\begin{eqnarray}\label{4.7}
\partial_i^2G_\lambda(x,t)=\frac{-1}{2(T-t+\lambda)}G_\lambda(x,t)+\frac{|x_i-x_{0i}|^2}{4(T-t+\lambda)^2}G_\lambda(x,t),
\end{eqnarray}
and for $i\neq j$,
\begin{eqnarray}\label{4.8}
\partial_i\partial_jG_\lambda(x,t)=\frac{(x_i-x_{0i})(x_j-x_{0j})}{4(T-t+\lambda)^2}G_\lambda(x,t),
\end{eqnarray}
where  $x_{0i}$'s are i-th coordinate component of $x_0$.
\end{lemma}
Straight calculation combining knowledge of $G_\lambda(x,t)$ in Lemma \ref{lemma4.1}, we have
\begin{lemma}\label{lemma4.2}
For each $\lambda>0$,  the following identities holds
 for  $t\in(0,T)$:
\begin{eqnarray}\label{4.9}
{d\over dt}H_\lambda(t) = -2D_{\lambda}(t)+2\int_{\Omega}u(\partial_tu-\triangle u)G_{\lambda}dx,
\end{eqnarray}
\begin{eqnarray}\label{2.10}
&&\frac{d}{dt}\ln {H_\lambda(t)}=-N_\lambda(t)+{2\over H_\lambda(t)}\int_\Omega u(\partial_tu-\triangle u) G_\lambda dx,
\end{eqnarray}
and
\begin{eqnarray}\label{4.15}
D'_\lambda(t)
&:=&-\theta-2\int_\Omega\left(\partial_tu-\frac{x-x_0}{2(T-t+\lambda)}\cdot\nabla u+\frac{1}{2}(\triangle u-\partial_tu)\right)^2G_\lambda dx\crr\disp
&&+\frac{1}{2}\int_\Omega(\triangle u-\partial_tu)^2G_\lambda dx
+\frac{1}{(T-t+\lambda)}D_\lambda(t),
\end{eqnarray}
where
\begin{eqnarray*}
\theta:=\int_{\partial\Omega}|\nabla u|^2\partial_\nu G_\lambda d\sigma-2\int_{\partial\Omega}\partial_\nu u(\nabla u\cdot\nabla G_\lambda)d\sigma.
\end{eqnarray*}
\end{lemma}
Direct calculation lead to (\ref{4.9}) --(\ref{4.15}).\\

\begin{lemma}\label{lemma4.5}
For $\lambda>0$ and $t\in(0,T)$,  frequency function $N_\lambda(t)$ holds
\begin{eqnarray}\label{4.17}
\frac{d}{dt}\bigg[(T-t+\lambda)\exp(-M^2t)N_\lambda(t)\bigg]\leqslant CM^2(T+\lambda).
\end{eqnarray}
\end{lemma}
\begin{proof}
\begin{eqnarray*}
N'_\lambda(t)&=&\frac{2}{H^2_\lambda(t)}\left[D'_\lambda(t)H_\lambda(t)-D_\lambda(t)H'_\lambda(t)\right]\crr\disp
&:=& I_1+I_2,
\end{eqnarray*} representing first term and second term; further,
\begin{eqnarray*}
I_1 &=&\frac{2}{H_\lambda(t)}\bigg\{-\theta-2\int_\Omega\left(\partial_tu-\frac{x-x_0}{2(T-t+\lambda)}\cdot\nabla u+\frac{1}{2}(\triangle u-\partial_tu)\right)^2G_\lambda dx+\frac{1}{2}\int_\Omega(\triangle u-\partial_tu)^2G_\lambda\crr\disp
&&+\frac{1}{(T-t+\lambda)}D_\lambda(t)\bigg\}H_{\lambda}(t)\crr\disp
&=&{2\over H_{\lambda}(t)} \left(-\theta +{D_{\lambda}(t)\over T-t+\lambda}\right) -{4\over H_{\lambda}(t)}\left[\int_\Omega\left(\partial_tu-\frac{x-x_0}{2(T-t+\lambda)}\cdot\nabla u+\frac{1}{2}(\triangle u-\partial_tu)\right)^2G_\lambda dx\right] \crr\disp
&& + {1\over H_{\lambda}(t)}\int_{\Omega}(\triangle u-\partial_t u)^2G_{\lambda}dx.
\end{eqnarray*}
and
\begin{eqnarray*}
I_2 &=&{1\over H^2_{\lambda}(t)}\bigg\{4\bigg[\int_\Omega u\left(\partial_tu-\frac{x-x_0}{2(T-t+\lambda)}\cdot\nabla u+\frac{1}{2}(\triangle u-\partial_t u)\right)G_\lambda dx\bigg]^2\crr\disp
&&-\left[\int_\Omega u(\triangle u-\partial_tu)G_\lambda dx\right]^2\bigg\}.\crr\disp
\end{eqnarray*} from Lemma \ref{lemma4.2}. Therefore,
\begin{eqnarray*}
N_{\lambda}'(t) &=& {2\over H_{\lambda}(t)}\left(-\theta + {1\over T-t+\lambda}D_{\lambda}(t)\right)+{1\over H_{\lambda}(t)}\int_{\Omega}(\triangle u-\partial_t u)^2G_\lambda dx\crr\disp
&&-{4\over H^2_{\lambda}(t)}\bigg[\int_\Omega u\left(\partial_tu-\frac{x-x_0}{2(T-t+\lambda)}\cdot\nabla u+\frac{1}{2}(\triangle u-\partial_tu)\right)G_\lambda dx\bigg]^2 \crr\disp
&&- {1\over H^2_{\lambda}(t)}\bigg(\int_{\Omega} u(\triangle u-\partial_t u)G_{\lambda}dx\bigg)^2\crr\disp
&&-{4\over H_{\lambda}(t)}\int_{\Omega}(\partial_tu-\frac{x-x_0}{2(T-t+\lambda)}\cdot\nabla u+\frac{1}{2}(\triangle u-\partial_tu))^2G_{\lambda}dx\crr\disp
&\le&\frac{1}{(T-t+\lambda)}N_\lambda(t)-2\frac{\theta}{H_\lambda(t)}+{1\over H_{\lambda}(t)}\int_{\Omega}(\triangle u-\partial_t u)^2G_{\lambda}dx,
\end{eqnarray*}
which leads to
\begin{eqnarray}\label{4.21}
N'_{\lambda} (t) -{1\over T-t+\lambda}N_{\lambda}(t) + {2\theta\over H_{\lambda}(t)} \le {1\over H_{\lambda}(t)}\int_{\Omega} (\triangle u - \partial_t u)^2G_{\lambda}dx.
\end{eqnarray}
On the right hand side,
\begin{eqnarray*}
{1\over H_{\lambda}(t)}\int_{\Omega} (\triangle u - \partial_t u)^2G_{\lambda}dx
&\le& {CM^2\over H_{\lambda}(t)} \int_{\Omega}|\nabla u|^2 G_{\lambda}dx + {2M^2\over H_{\lambda}(t)} \int_{\Omega} |u|^2G_{\lambda}dx.
\end{eqnarray*}
 Hence,
\begin{eqnarray*}
N'_{\lambda} (t) - {N_{\lambda}(t)\over T-t+\lambda} + {2\theta\over H_{\lambda}(t)} -CM^2N_{\lambda}(t) \le CM^2.
\end{eqnarray*}
Since  the domain
 is convex, we have $\theta\ge 0.$  Thus,
\begin{eqnarray}\label{4.22}
N'_{\lambda}(t) - \big({1\over {T-t+\lambda}}+CM^2 \big)N_{\lambda}(t) \le CM^2.
\end{eqnarray}
Therefore, by multiplying integral factor $\exp\big(\ln(T-t+\lambda)-M^2t\big)$, for any $t\in(0,T)$, (\ref{4.22}) can be written as
\begin{eqnarray*}
{d\over dt}\bigg[N_{\lambda}(t)(T-t+\lambda)\exp(-M^2t)\bigg] \le CM^2(T-t+\lambda),
\end{eqnarray*} which leads to the conclusion. \end{proof}

Let a constant $K_T$ be
\begin{eqnarray}\label{4.23}
K_{T}=4\ln\bigg({\int_{\Omega}|u(x, 0)|^2dx\over\int_{\Omega}|u(x, T)|^2dx}\bigg)+{2m\over T} +  CM^2T^2+CMT +\frac{n}{2}.
\end{eqnarray}
\begin{lemma}\label{lemma4.6}
For each $\lambda>0$,  it holds that:
\begin{eqnarray}\label{4.24}
\lambda e^{-M^2T} N_\lambda(T)+\frac{n}{2}\leq\left(\frac{\lambda}{T}+1\right)K_T.
\end{eqnarray}
\end{lemma}
\begin{proof}
 Integrating (\ref{4.17}) over $(t, T)$, we infer
\begin{eqnarray*}
\lambda e^{-M^2T}N_{\lambda}(T) - (T-t+\lambda)e^{-M^2t}N_{\lambda}(t)\le CM^2(T+\lambda)T,
\end{eqnarray*}
integrating the above on $(0,\,\,{T\over 2})$, we get
\begin{eqnarray*}
{T\over 2}\lambda e^{-M^2T}N_{\lambda}(T) \le (T+\lambda)\int_0^{T\over 2} N_{\lambda}(t)dt + M^2{T^2\over 2}(T+\lambda).
\end{eqnarray*}
Since Lemma \ref{lemma4.2}, we have
\begin{eqnarray*}
&&\int_{0}^{T\over 2} N_{\lambda}(t)dt = -\int_0^{T\over 2} {H_{\lambda}'(t)\over H_{\lambda}(t)}dt + \int_0^{T\over2}{2\over H_{\lambda}(t)}\int_{\Omega}u(\partial_tu-\triangle u)G_{\lambda}dxdt\\
&&\quad= -\ln({H_{\lambda}({T\over 2})\over H_{\lambda}(0)})+ \int_0^{T\over2}{2\over H_{\lambda}(t)}\int_{\Omega}u(\partial_tu-\triangle u)G_{\lambda}dxdt.
\end{eqnarray*}
On the right hand side,
\begin{eqnarray*}
&&\int_0^{T\over2}{2\over H_{\lambda}(t)}\int_{\Omega}u(\partial_tu-\triangle u)G_{\lambda}dxdt.\\
&\le& \int_0^{T\over2} {CM\over H_{\lambda}(t)}\int_{\Omega}(|\nabla u|+|u|)|u|G_{\lambda}dxdt\\
&\le& \int_0^{T\over2} {CM\over H_{\lambda}(t)}(\int_{\Omega}|\nabla u||u|G_{\lambda}dxdt+\int_{\Omega}|u|^2G_{\lambda}dxdt)\\
&:=& I_3+I_4.
\end{eqnarray*}
Estimating $I_3$ and $I_4$, it follows that
\begin{eqnarray*}
I_3:= \int_0^{T\over2} {CM\over H_{\lambda}(t)}\int_{\Omega}|\nabla u||u|G_{\lambda}dxdt
\leq {1\over 2}\int_0^{T\over 2}{2\int_{\Omega}|\nabla u|^2G_{\lambda}dx\over H_{\lambda}(t)}dt + CMT,
\end{eqnarray*}
and
\begin{eqnarray*}
I_4&:=&\int_0^{T\over 2}{CM\over H_{\lambda}(t)}\int_{\Omega}u^2 G_{\lambda}dxdt\leq CMT.
\end{eqnarray*}
In $I_3$, the term ${1\over 2}\int_{0}^{T\over 2}N_{\lambda}(t)dt$ can be moved to the left for combination.
Hence,
\begin{eqnarray*}
{T\over 2}\lambda e^{-M^2T}N_{\lambda}(t) \le 2(T+\lambda)\bigg[\ln{H_{\lambda}(0)\over H_{\lambda}({T\over 2})}+CMT\bigg] +CM^2{T^2\over 2}(T+\lambda).
\end{eqnarray*}
In the term $\ln{H_{\lambda}(0)\over H_{\lambda}({T\over 2})}$,
\begin{eqnarray*}
\frac{H_\lambda(0)}{H_\lambda(\frac{T}{2})}&=&{\int_\Omega|u(x,0)|^2(T+\lambda)^{-\frac{d}{2}}\cdot e^{-\frac{|x-x_0|^2}{4(T+\lambda)}}dx\over\int_\Omega|u(x,\frac{T}{2})|^2(\frac{T}{2}+\lambda)^{-\frac{d}{2}}\cdot e^{-\frac{|x-x_0|^2}{4(\frac{T}{2}+\lambda)}}dx}\crr\disp
&\leq&\frac{\int_\Omega|u(x,0)|^2dx\cdot(\frac{T}{2}+\lambda)^{\frac{d}{2}}}{\int_\Omega|u(x,\frac{T}{2})|^2\cdot e^{-\frac{|x-x_0|^2}{4(\frac{T}{2}+\lambda)}}dx\cdot(T+\lambda)^{\frac{d}{2}}}\crr\disp
&\leq&\frac{\int_\Omega|u(x,0)|^2dx\cdot(\frac{T}{2}+\lambda)^{\frac{d}{2}}}{\int_\Omega|u(x,\frac{T}{2})|^2 dx\cdot(T+\lambda)^{\frac{d}{2}}}\cdot e^{\frac{m}{4(\frac{T}{2}+\lambda)}}\crr\disp
&\leq&e^{\frac{m}{2T}}\frac{\int_\Omega|u(x,0)|^2dx}{\int_\Omega|u(x,\frac{T}{2})|^2dx}.
\end{eqnarray*}
Therefore,
\begin{eqnarray} \label{4.25}
&&{T\over 2}\lambda e^{-M^2T}N_{\lambda}(t) \le 2(T+\lambda)({m\over 2T}+\ln\big(\frac{\int_\Omega|u(x,0)|^2dx}{\int_\Omega|u(x,\frac{T}{2})|^2dx}\big)
+CMT)+ CM^2{T^2\over 2}(T+\lambda).
\end{eqnarray}
Using Assumption 1, we have
\begin{eqnarray}\label{4.26}
{\int_{\Omega }|u(x, T)|^2dx\over \int_{\Omega} |u(x, {T\over 2})|^2dx}\le \exp(CMT).
\end{eqnarray}
From (\ref{4.26}), we can move the ratio term on the left to the right side, then add the result to (\ref{4.25}), and obtain
\begin{eqnarray*}
\lambda e^{-M^2T}N_{\lambda}(T) &\le&({\lambda\over T}+1)\bigg[{2m\over T} + 4\ln\bigg({\int_{\Omega}|u(x, 0)|^2dx\over\int_{\Omega}|u(x, T)|^2dx}\bigg) + CMT+CM^2T^2\bigg]\\
&\le&({\lambda\over T} + 1)(K_T-{n\over 2})\\
&\le&({\lambda\over T} + 1)K_T-{n\over 2},
\end{eqnarray*}
and we obtain the result.
\end{proof}

\begin{lemma}\label{lemma4.7}
For $T>0$, the following estimate holds
\begin{eqnarray*}\label{4.28}
&&\left[1-\frac{8e^{M^2T}\lambda}{r^2}\left(\frac{\lambda}{T}+1\right)K_{T}\right]\int_\Omega|x-x_0|^2|u(x, T)|^2e^{-\frac{|x-x_0|^2}{4\lambda}}dx\nonumber\\
&&\quad\leq 8e^{M^2T}\lambda\left(\frac{\lambda}{T}+1\right){K}_{T}\int_{B_r}|u(x, T)|e^{-\frac{|x-x_0|^2}{4\lambda}}dx.
\end{eqnarray*}
\end{lemma}
\begin{proof}
We borrow a inequality from \cite{EscauriazaF1} (as well as in \cite{Phung0}) that for any $f \in H_0^1(\Omega)$ and for a $\lambda>0$,
\begin{eqnarray}\label{4.28}
\int_\Omega\frac{|x-x_0|^2}{8\lambda}|f(x)|^2e^{-\frac{|x-x_0|^2}{4\lambda}}dx&\leq&2\lambda\int_\Omega|\nabla f(x)|^2e^{-\frac{|x-x_0|^2}{4\lambda}}dx+\frac{n}{2}\int_\Omega|f(x)|^2e^{-\frac{|x-x_0|^2}{4\lambda}}dx.\nonumber\\
\end{eqnarray}
From this fact,
\begin{eqnarray*}
&&\int_\Omega|x-x_0|^2|u(x,T)|^2e^{-\frac{|x-x_0|^2}{4\lambda}}dx\crr\disp
&\leq&8\lambda\bigg(2\lambda\int_\Omega|\nabla u(x,T)|^2 e^{-\frac{|x-x_0|^2}{4\lambda}}dx+\frac{n}{2}\int_\Omega|u(x,T)|^2e^{-\frac{|x-x_0|^2}{4\lambda}}dx\bigg)\crr\disp
&\leq&8\lambda\left(\lambda N_\lambda(T)+\frac{n}{2}\right)\int_\Omega|u(x,T)|^2e^{-\frac{|x-x_0|^2}{4\lambda}}dx\crr\disp
&\leq&8\lambda\left(\lambda N_\lambda(T)+\frac{n}{2}\right)\bigg[\int_{B_r}|u(x,T)|^2e^{-\frac{|x-x_0|^2}{4\lambda}}dx+\frac{1}{r^2}\int_{\Omega\setminus B_r}|x-x_0|^2|u(x,T)|^2e^{-\frac{|x-x_0|^2}{4\lambda}}dx\bigg],
\end{eqnarray*} as ${|x-x_0|\over r} \ge 1$ when $x\not\in B_r$.\\

With help of (\ref{4.24}), we can observe
\begin{eqnarray*}
\int_\Omega|x-x_0|^2|u(x,T)|^2e^{-\frac{|x-x_0|^2}{4\lambda}}dx&\leq&8\lambda e^{M^2T}\left(\frac{\lambda}{T}+1\right)K_{T}\bigg[\int_{B_r}|u(x,T)|^2e^{-\frac{|x-x_0|^2}{4\lambda}}dx\crr\disp
&&+\frac{1}{r^2}\int_\Omega|x-x_0|^2|u(x,T)|^2e^{-\frac{|x-x_0|^2}{4\lambda}}dx\bigg].
\end{eqnarray*}
We can arrive at the result from the following
\begin{eqnarray*}
&&\left[1-\frac{8e^{M^2T}\lambda}{r^2}\left(\frac{\lambda}{T}+1\right)K_{T}\right]\int_\Omega|x-x_0|^2|u(x, T)|^2e^{-\frac{|x-x_0|^2}{4\lambda}}dx\nonumber\\
&&\quad\leq 8e^{M^2T}\lambda\left(\frac{\lambda}{T}+1\right){K}_{T}\int_{B_r}|u(x, T)|e^{-\frac{|x-x_0|^2}{4\lambda}}dx.
\end{eqnarray*}
 This complete the proof.
\end{proof}

\section{The unique continuation  property}\label{sec2}
\subsection{Proof of Theorem \ref{theorem1.1} }
\begin{proof}
We take $$\lambda = {1\over 2}\left({-T}+\sqrt{T^2+{r^2T\over 4K_Te^{M^2T}}}\right)>0$$ such that
\begin{eqnarray}\label{4.3.1}
\frac{8e^{M^2T}\lambda}{r^2}\left(\frac{\lambda}{T}+1\right){K}_{T}=\frac{1}{2}.
\end{eqnarray}
It follows that
\begin{eqnarray*}
\frac{1}{\lambda}&=&2\frac{T+\sqrt{T^2+\frac{Tr^2}{4e^{M^2T}{K}_{T}}}}{\frac{Tr^2}{4e^{M^2T}{K}_{T}}}\crr\disp
&=&8\left(T+\sqrt{T^2+\frac{Tr^2}{4e^{M^2T}{K}_{T}}}\right)\frac{1}{Tr^2}e^{M^2T}{K}_{T}\crr\disp
&\leq&8\left(2T+\sqrt{\frac{Tr^2}{4e^{M^2T}{K}_{T}}}\right)\frac{1}{Tr^2}e^{M^2T}{K}_{T}\crr\disp
&\leq&\left(16+\frac{4r}{\sqrt{m}}\right)\frac{1}{r^2}e^{M^2T}{K}_{T},
\end{eqnarray*} as ${m\over T}\le K_T$ and $e^{M^2T}\ge 1$.
We obtain
\begin{eqnarray}\label{4.3.2}
e^{\frac{m}{4\lambda}}&\leq&e^{(4m+r\sqrt{m})\frac{1}{r^2}{K}_{T}e^{M^2T}}\crr\disp
&\leq&e^{(4m+r\sqrt{m})\frac{1}{r^2}\frac{n}{2}e^{M^2T}}e^{(4m+r\sqrt{m})\frac{1}{r^2}e^{M^2T}({2m\over T} +  CM^2T^2+CMT) }\crr\disp
&&\times\left(\frac{\int_\Omega|u(x,0)|^2dx}{\int_\Omega|u(x,T)|^2dx}\right)^{{4\over r^2}(4m+r\sqrt{m})e^{M^2T}}.
\end{eqnarray}
From Lemma \ref{lemma4.7}, we have
\begin{eqnarray}\label{4.3.3}
&&\int_\Omega|x-x_0|^2|u(x,T)|^2e^{-\frac{|x-x_0|^2}{4\lambda}}dx\nonumber\\
&&\quad\leq r^2\int_{B_r}|u(x,T)|^2e^{-\frac{|x-x_0|^2}{4\lambda}}dx.
\end{eqnarray}
Combining (\ref{4.3.2}) and (\ref{4.3.3}),
\begin{eqnarray*}
&&\int_\Omega|u(x,T)|^2e^{-\frac{m}{4\lambda}}dx\leq\int_\Omega|u(x,T)|^2e^{-\frac{|x-x_0|^2}{4\lambda}}dx\crr\disp
&\leq&\int_{\Omega\backslash B_r}|u(x,T)|^2e^{-\frac{|x-x_0|^2}{4\lambda}}dx+\int_{B_r}|y(x,T)|^2e^{-\frac{|x-x_0|^2}{4\lambda}}dx\crr\disp
&\leq&\frac{1}{r^2}\int_\Omega|x-x_0|^2|u(x,T)|^2e^{-\frac{|x-x_0|^2}{4\lambda}}dx+\int_{B_r}|u(x,T)|^2e^{-\frac{|x-x_0|^2}{4\lambda}}dx\crr\disp
&\leq&2\int_{B_r}|u(x,T)|^2e^{-\frac{|x-x_0|^2}{4\lambda}}dx\crr\disp
&\leq&2\int_{B_r}|u(x,T)|^2dx,
\end{eqnarray*}
hence
\begin{eqnarray*}
\int_\Omega|u(x,T)|^2dx&\leq&2e^{\frac{m}{4\lambda}}\int_{B_r}|u(x,T)|^2dx\crr\disp
&\leq&2e^{(4m+r\sqrt{m})\frac{1}{r^2}\frac{n}{2}e^{M^2T}}e^{(4m+r\sqrt{m})\frac{1}{r^2}({2m\over T} +  CM^2T^2+CMT )e^{M^2T}}\crr\disp
&&\times\left(\frac{\int_\Omega|u(x,0)|^2dx}{\int_\Omega|u(x,T)|^2dx}\right)^{{4\over r^2}(4m+r\sqrt{m})e^{M^2T}}\times\int_{B_r}|u(x,T)|^2dx.
\end{eqnarray*}
Thus, we have
\begin{eqnarray*}
\int_\Omega|u(x,T)|^2dx&\leq& 2e^{Ce^{M^2T}\over r^2}e^{\frac{C}{r^2}[{2m\over T}+ M^2T^2+MT]e^{M^2T}}\left(\frac{\int_\Omega|u(x,0)|^2dx}{\int_\Omega|u(x,T)|^2dx}\right)^{Ce^{M^2T}\over r^2}\crr\disp
&&\times\int_{B_r}|u(x,T)|^2dx.
\end{eqnarray*}
This is equivalent to the following inequality:
\begin{eqnarray}\label{4.3.4}
\int_\Omega|u(x,T)|^2dx&\leq&Ce^{C[{2m\over T}+ M^2T^2+MT]}\left(\int_\Omega\left(|u(x,0)|^2\right)dx\right)^{\frac{C'}{r^2+C'}}\crr\disp
&&\times\left(\int_{B_r}|u(x,T)|^2dx\right)^{\frac{r^2}{r^2+C'}},
\end{eqnarray}
where $C'=4(4m+r\sqrt{m})e^{M^2T}$.
Let $\gamma=\frac{r^2}{r^2+C'}$, and the above estimate gives
\begin{eqnarray*}
\int_\Omega|u(x,T)|^2dx\leq Ce^{C[{2m\over T}+ M^2T^2+MT]}\left(\int_\Omega(|u(x,0)|^2dx\right)^{1-\gamma}\left(\int_\omega(|u(x,T)|^2dx\right)^\gamma.
\end{eqnarray*}
 This complete the proof.
\end{proof}

\subsection{Proof of Theorem \ref{theorem1.3}}

\begin{proof}
We will first prove a backward uniqueness estimate:
\begin{eqnarray}\label{4.29}
 \|u(0)\|_{H^{-1}(\Omega)}^2\le \exp(2e^{CM^2T} \left(\zeta(0) + CM\sqrt{\zeta(0)}\right)T)\|u(T)\|_{H^{-1}(\Omega)}^2.
\end{eqnarray}
We first address two energy estimates  from  equation (\ref{1.1}).
Multiplying $u$ and $(-\triangle)^{-1}u$ with $\partial_tu-\triangle u$, we have two energy identities
\begin{eqnarray}\label{3.6}
&&{1\over 2}{d\over dt}\|u(t)\|_2^2 + \|u(t)\|^2_{H^1_0(\Omega)} = \langle(\partial_tu-\triangle u), u\rangle,\nonumber\\
&&{1\over 2}{d\over dt}\|u(t)\|_{H^{-1}(\Omega)}^2 + \|u(t)\|_2^2 = \langle(\partial_tu-\triangle u), (-\triangle)^{-1}u\rangle.
\end{eqnarray}
Let $f =\partial_tu-\triangle u,$ and $\zeta(t): = {\|u(t)\|_2^2\over \|u(t)\|^2_{H^{-1}}},$ then
\begin{eqnarray*}
\zeta'(t) &=& {2 \over \|u\|_{H^{-1}(\Omega)}^4}\left({\langle f, u\rangle\|u\|^2_{H^{-1}(\Omega)}-\|u\|^2_{H_0^1(\Omega)}\|u\|^2_{H^{-1}(\Omega)} - \langle f, (-\Delta)^{-1}u\rangle\|u\|_2^2 +\|u\|_2^4}\right).
\end{eqnarray*}
By direct computation, we have
\begin{eqnarray*}
&&\|u\|^4 - \|u\|^2\langle f, (-\triangle)^{-1}u\rangle = |\langle\triangle u+ {f\over 2}, (-\triangle)^{-1}u\rangle|^2-|\langle{f\over 2}, (-\triangle)^{-1}u\rangle|^2 \nonumber \\
&\leq&\big(\|u\|^2_{H_0^1(\Omega)}+\|{f\over 2}\|_{H^{-1}(\Omega)}^2-\langle f,u\rangle\big)\|u\|^2_{H^{-1}(\Omega)}-|\langle{f\over 2}, (-\triangle)^{-1}u\rangle|^2.
\end{eqnarray*}
Therefore, we can obtain the following estimate:
\begin{eqnarray*}
{d\over dt}\zeta(t)\le {2\over \|u\|^2_{H^{-1}(\Omega)}}\|{f\over 2}\|_{H^{-1}(\Omega)}^2.
\end{eqnarray*}
By Assumption 2,
 \begin{eqnarray}\label{3.7a}
\|f\|_{H^{-1}(\Omega)}\le CM\|u\|_{2},
\end{eqnarray} thus
 \begin{eqnarray}\label{3.7}
 \zeta(t)\le e^{CM^2t}\zeta(0).
 \end{eqnarray}
 Applying this to $H^{-1}$ energy identity (\ref{3.6}), we have
 \begin{eqnarray*}
 &&0\le {1\over 2}{d\over dt}\|u(t)\|_{H^{-1}(\Omega)}^2 +\zeta(t) \|u(t)\|_{H^{-1}(\Omega)}^2 + |\langle f, (-\triangle)^{-1}u\rangle|\\
 &&\quad \le {1\over 2}{d\over dt}\|u(t)\|_{H^{-1}(\Omega)}^2 +\zeta(t) \|u(t)\|_{H^{-1}(\Omega)}^2 + \|f\|_{H^{-1}(\Omega)}\|u\|_{H^{-1}(\Omega)}\\
 &&\quad\le {1\over 2}{d\over dt}\|u(t)\|_{H^{-1}(\Omega)}^2 +\zeta(0)e^{CM^2t} \|u(t)\|_{H^{-1}(\Omega)}^2 + CM\|u\|_{2}\|u\|_{H^{-1}(\Omega)}\\
 &&\quad\le {1\over 2}{d\over dt}\|u(t)\|_{H^{-1}(\Omega)}^2 +\zeta(0)e^{CM^2t}  \|u(t)\|_{H^{-1}(\Omega)}^2 +CM\sqrt{\zeta(t)}\|u\|_{H^{-1}(\Omega)}^2\\
 &&\quad\le  {1\over 2}{d\over dt}\|u(t)\|_{H^{-1}(\Omega)}^2 +e^{CM^2T} \left(\zeta(0) + CM\sqrt{\zeta(0)}\right)\|u(t)\|_{H^{-1}(\Omega)}^2.
 \end{eqnarray*}
 Therefore,
 \begin{eqnarray*}
 \|u(0)\|_{H^{-1}(\Omega)}^2\le \exp(2e^{CM^2T} \left(\zeta(0) + CM\sqrt{\zeta(0)}\right)T)\|u(T)\|_{H^{-1}(\Omega)}^2.
 \end{eqnarray*}
This yields
 \begin{eqnarray*}
&& \frac{\|u(0)\|_{2}^2}{\|u(T)\|_{2}^2}\leq \frac{\|u(0)\|_{H^{-1}(\Omega)}^2}{\|u(T)\|_{H^{-1}(\Omega)}^2}\zeta(0)\nonumber\\
&\leq& \zeta(0)\exp(2e^{CM^2T} \left(\zeta(0) + CM\sqrt{\zeta(0)}\right)T)\nonumber\\
&\leq& \exp\big(C(1+MT)e^{CM^2T} \zeta(0)\big).
 \end{eqnarray*}
 This, together with (\ref{1.4}), deduces (\ref{1.5}).
 This completes the
proof.
\end{proof}

\section{A nontrivial example}
 In this section, we will give a nontrivial parabolic case. We consider the equation as follows:
\begin{equation}
\begin{cases}
\partial_t u-\triangle u+\sum_{i=1}^nb_i(x,t)\partial_iu+c(x,t)u=0,& \textrm{ in } \Omega\times(0,T),\\
u=0,&  \textrm{ on } \partial\Omega\times(0,T),\\
u(x,0)=u_0,
\end{cases}
\end{equation}
where  $u$ denote states $u(x, t)$  at spatial position
$x\in\Omega$ and time $t\geq0$, and
 the initial data $u_0(x)\in L^2(\Omega)$. \\

Now, we suppose that the  coefficients $b_i(x,t), c(x,t)$, $(i=1,2,\ldots,n)$ satisfy
\begin{eqnarray}\label{condition1}
b_i(x,t),c(x,t)\in L^{\infty}(\Omega\times(0,T)), (i=1,2,\ldots,n),
\end{eqnarray}
and
 \begin{eqnarray}\label{condition2}
M=\max\{\|b_i\|_{L^{\infty}(\Omega\times(0,T))},\|c\|_{L^{\infty}(\Omega\times(0,T))}\:|\:i=1,2\ldots,n\}.
\end{eqnarray}
Thus, the solutions of equation (\ref{1.1})  $u(x,t)\in L^2(0,T; H_0^1(\Omega))\cap H^1(0,T; H^{-1}(\Omega))$ as the
 initial data  $u_0\in L^2(\Omega)$. This is  Assumption 1 in Section 1.  By the standard energy estimate, we obtain Assumption 2 holds in this case.
In order to get Assumption 2, we will introduce the following Lemma.
\begin{lemma}\label{lemma4.a}
Suppose that $h\in L^\infty(\Omega)$, and $g\in H_0^1(\Omega)$. Then, we have
\begin{eqnarray*}
\|h\cdot \partial_ig\|_{H^{-1}(\Omega)}\leq C\|h\|_{L^\infty(\Omega)}\cdot\|g\|_{2},\: i=1,2,\ldots,n.
\end{eqnarray*}
\end{lemma}
\begin{proof} We will prove for any fixed $i\in\{1,\,\,2,\,\,\cdots,n\}$. Clearly we can find a function $v\in L^\infty(\Omega)$ such that  $\partial_iv=h$, and $\|v\|_{L^\infty(\Omega)}\leq C\|h\|_{L^\infty(\Omega)}$.
Then, $h\cdot \partial_ig=\partial_i(v\cdot g)-h\cdot g$.
Thus,
\begin{eqnarray*}
\|h\cdot \partial_ig\|_{H^{-1}(\Omega)}&\leq& \|\partial_i(v\cdot g)\|_{H^{-1}(\Omega)}+\|h\cdot g\|_{H^{-1}(\Omega)}\nonumber \\
&\leq&  \| v\cdot g \|_{2}+\|h\cdot g\|_{2}\\
&\leq& C\|h\|_{L^\infty(\Omega)}\cdot\|g\|_{2}.
\end{eqnarray*}  This complete the proof.
\end{proof}
By this Lemma, Assumption 3 also holds in this case.

\section*{Acknowledgement}

Authors are grateful for referees' comments and guidances.

\end{document}